\documentclass[11pt,english]{smfart}

\usepackage[T1]{fontenc}
\usepackage[english,francais]{babel}
\usepackage{amssymb,url,xspace,smfthm}
\usepackage[all]{xy}
\usepackage[dvips]{graphicx}

\makeatletter
    \def\ps@copyright{\ps@empty
    \def\@oddfoot{\hfil\small\copyright 2012, \SMF}}
\makeatother

\newcommand{\SMF}{Soci\'et\'e Ma\-th\'e\-Ma\-ti\-que de France}
\newcommand{\BibTeX}{{\scshape Bib}\kern-.08em\TeX}
\newcommand{\T}{\S\kern .15em\relax }
\newcommand{\AMS}{$\mathcal{A}$\kern-.1667em\lower.5ex\hbox
        {$\mathcal{M}$}\kern-.125em$\mathcal{S}$}

\input yoga.sty

\newtheorem*{maintheo}{{Theorem}}

\newcommand{\ra}{\rightarrow}

\newcommand{\ul}{\underline}

\newcommand{\rA}{\mathrm{A}}

\newcommand{\rD}{\mathrm{D}}
\newcommand{\rG}{\mathrm{G}}
\newcommand{\rE}{\mathrm{E}}
\newcommand{\rH}{\mathrm{H}}
\newcommand{\rB}{\mathrm{B}}
\newcommand{\rC}{\mathrm{C}}
\newcommand{\rT}{\mathrm{T}}
\newcommand{\rW}{\mathrm{W}}
\newcommand{\rX}{\mathrm{X}}
\newcommand{\rV}{\mathrm{V}}
\newcommand{\rS}{\mathrm{S}}
\newcommand{\rU}{\mathrm{U}}
\newcommand{\rM}{\mathrm{M}}
\newcommand{\rN}{\mathrm{N}}
\newcommand{\rR}{\mathrm{R}}

\newcommand{\rQ}{\mathrm{Q}}
\newcommand{\rI}{\mathrm{I}}
\newcommand{\rp}{\mathrm{p}}
\newcommand{\e}{\mathrm{e}}
\newcommand{\lieg}{\mathfrak{g}}
\newcommand{\liet}{\mathfrak{t}}
\newcommand{\ent}{\mathbb{Z}}
\newcommand{\aff}{\mathbb{A}}
\newcommand{\rat}{\mathbb{Q}}
\newcommand{\compl}{\mathbb{C}}

\newcommand{\rGa}{\mathbb{G}_{\mathrm{a}}}
\newcommand{\rGm}{\mathbb{G}_{\mathrm{m}}}
\newcommand{\ratc}{\overline{\mathbb{Q}}}

\newcommand{\cL}{\mathcal{L}}

\newcommand{\Sym}{\mathrm{Sym}}

\newcommand{\Img}{\mathrm{Im}}
\newcommand{\tor}{\mathrm{Tor}}
\newcommand{\Dim}{\mathrm{dim}}
\newcommand{\Isom}{\underline{\mathrm{Isom}}}
\newcommand{\Isomext}{\underline{\mathrm{Isomext}}}
\newcommand{\adqG}{\mathrm{G}//\mathrm{G}}
\newcommand{\spec}{\mathrm{Spec}}

\newcommand{\rGL}{\mathrm{GL}}
\newcommand{\rPGL}{\mathrm{PGL}}
\newcommand{\Hom}{\mathrm{Hom}}
\newcommand{\Aut}{\mathrm{Aut}}
\newcommand{\ad}{\mathrm{ad}}

\newcommand{\Out}{\mathrm{Out}}
\newcommand{\Nor}{\mathrm{Norm}}
\newcommand{\Cent}{\mathrm{Centr}}
\newcommand{\Dyn}{\underline{\mathrm{Dyn}}}
\newcommand{\expo}{\mathrm{exp}}
\newcommand{\cV}{\mathcal{V}}
\newcommand{\cW}{\mathcal{W}}
\newcommand{\cE}{\mathscr{E}}
\newcommand{\cO}{\mathscr{O}}

\title{Adjoint quotients of reductive groups}
\date{\today}
\author{Ting-Yu Lee}

\address{DMA-Ecole normal sup\'erieure \\
45 Rue d'Ulm, F-75230 Paris Cedex 05, France}
\email{Ting-Yu.Lee@ens.fr} 

\begin{document}
\def\smfbyname{}
\newenvironment{fact}{\begin{enonce}{Fact}}{\end{enonce}}
\begin{abstract}
Let $\rG$ be a reductive group over a commutative ring $k$. In this
article, we prove that the adjoint quotient $\adqG$ is stable under
base change. Moreover, if $\rG$ has a maximal torus $\rT$, then the
adjoint quotient of the torus $\rT$ by its Weyl group will be
isomorphic to $\adqG$. Then we focus on the semisimple simply
connected group $\rG$ of the constant type. In this case, $\adqG$ is
isomorphic to the Weil restriction $\underset{\rD/\spec
k}{\prod}\aff^{1}_\rD$, where $\rD$ is the Dynkin scheme of $\rG$.
Then we prove that for such $\rG$, the Steinberg's cross-section can
be defined over $k$ if $\rG$ is quasi-split and without
$\rA_{2m}$-type components.
\end{abstract}

\begin{altabstract}
Soit $k$ un anneau commutatif et $\rG$ un groupe r\'eductif sur $k$.
Dans cet article, on va definir le quotient adjoint $\adqG$ de $\rG$
sur $k$ et d\'emontrer que la construction est stable par changement
de base. En plus, si $\rG$ poss\`{e}de un tore maximal $\rT$, le
quotient adjoint de $\rT$ par son groupe de Weyl est isomorphe \`{a}
$\adqG$. Dans la derniere section, on se concentre sur le cas $\rG$
semi-simple simplement connexe de type constant. Dans ce cas,
$\adqG$ est isomorphe \`{a} la restriction de Weil
$\underset{\rD/\spec k}{\prod}\aff^{1}_\rD$, o\`{u} $\rD$ est le
sch\'{e}ma de Dynkin. Si $\rG$ est de plus quasi-d\'{e}ployable et
sans composantes de type $\rA_{2m}$, on peut construire la
cross-section de Steinberg sur $k$.

\end{altabstract}

\maketitle


\section{Introduction} Let $k$ be a commutative ring
and $\rG$ be a reductive group over $k$. In this article, we want to
discuss the adjoint quotient of $\rG$ which is denoted by $\adqG$.

Roughly speaking, the adjoint quotient of $\rG$ is determined by
those regular functions of $\rG$
which are constants on the conjugacy classes of $\rG$. 
Suppose that $\rG$ contains a maximal torus $\rT$ and let $\rW$ be
the corresponding Weyl group. Then the $\rG$-conjugation action on
the regular functions of $\rG$ induces a $\rW$-conjugation action on
the regular functions of $\rT$. Let $\rT//\rW$ be the adjoint
quotient of $\rT$ by $\rW$. The natural restriction on regular
functions induces a natural morphism $$\iota:\rT//\rW\ra \adqG.$$
When $k$ is an algebraically closed field, $i^{\sharp}$ is an
isomorphism. One can find the classical treatment about adjoint
quotients over algebraically closed fields in Steinberg's
paper~\cite{St} \S 6, or in Humphreys's book~\cite{Hum} Chap. 3.

In this article, we will show that the same result holds for any
commutative ring $k$. Namely, we will prove the following theorem:
\begin{maintheo}
Let $k$ be a commutative ring and $\rG$ be a reductive group defined
over $k$. Suppose that $\rG$ contains a maximal $k$-torus $\rT$. Let
$\rW$ be the corresponding Weyl group of $\rT$. Then
$\rT//\rW\xrightarrow{\sim} \adqG$.
\end{maintheo}

The strategy we take here is actually the same one used for k an
algebraically closed field. In $\S$ 3, we will prove the Theorem
over $\ent$ and generalize the result to arbitrary commutative rings
in $\S$ 4.

In $\S$ 5, we focus on the adjoint quotient of the semisimple simply
connected group scheme of constant type. For such group $\rG$,
$\adqG$ is isomorphic to the Weil restriction $\underset{\rD/\spec
k}{\prod}\aff^{1}_\rD$, where $\rD$ is the Dynkin scheme of $\rG$.
Moreover, we prove that the Steinberg's cross-section (\cite{St},
Thm. 1.4) can be defined over arbitrary commutative ring $k$ if
$\rG$ is quasi-split and without $\rA_{2m}$-type components. At the
end of the article, we prove that for a semisimple simply connected
group scheme $\rG$ of constant type, it always contains a semisimple
regular element over a semi-local ring $k$, and the centralizer of a
semisimple regular element is a maximal torus of $\rG$.


\begin{section}{Notations and Definitions}
Let $k$ be a commutative ring. For an affine $k$-scheme $\rX$, we
let $k[\rX]$ be the ring of regular functions of $\rX$. For a
$k$-scheme $\rX$, a $k$-algebra $\rA$, we let $\rX_{\rA}$ be the
fiber product $\rX\times_{k}\spec\rA$. Let $\rGm$ (resp. $\rGa$) be
the multiplicative (resp. additive) group defined over $\ent$.

For a $k$-module $\rV$, we regard it as a functor by defining
\begin{center}
$\rV(\rA)=\rV\otimes_{k}\rA$, for all $k$-algebras $\rA$.
\end{center}
In order to define the adjoint quotient of $\rG$ over an arbitrary
commutative ring $k$, we first define a $\rG$-conjugation action on
the $k$-module $\rV:=k[\rG]$. Let $\rA$ be a $k$-algebra and
$g\in\rG(\rA)$, $f\in\rV(\rA)$, we define
\begin{center}
$(g.f)(x)=f(gxg^{-1})$, for all $\rA$-algebras $\rA^{\prime}$, and
for all $x\in\rG(\rA^{\prime})$.
\end{center}


Let $c:\rV\ra k[\rG]\otimes\rV$ be the comodule map corresponding to
the conjugation action defined above. We define
\begin{eqnarray*}
&&k[\rG]^{\rG}:=\{f\in\rV(k)|\ \sigma. f_\rA=f_\rA,\ \forall\
\sigma\in\rG(\rA),\ \forall\ k-algebras\ \rA\},
\end{eqnarray*}
Let $\adqG=\spec (k[\rG]^{\rG})$  be the \emph{adjoint quotient} of
$\rG$. Suppose that $\rG$ contains a maximal torus $\rT$ and let
$\rW$ be the corresponding Weyl group. Then the $\rG$-conjugation
action on $k[\rG]$ induces a $\rW$-conjugation action on $k[\rT]$
and let $\rT//\rW=\spec (k[\rT]^{\rW}$). The natural restriction
from $k[\rG]$ to $k[\rT]$ induces a natural homomorphism
$$i:k[\rG]^{\rG}\ra k[\rT]^{\rW}.$$ When $k$ is an algebraically
closed field, $i$ is an isomorphism. Namely,
\begin{theo}\label{0:3}
Let $k$ be an algebraically closed field. Let $\rG$ be a semisimple
$k$-group and $\rT$ be a maximal torus of $\rG$. Let $\rW$ be the
Weyl group with respect to $\rT$. Then the restriction map
$i:k[\rG]^{\rG}\ra k[\rT]^{\rW}$ is an isomorphism. Furthermore, if
$\rG$ is semisimple simply connected, then $k[\rG]^{\rG}$ is freely
generated as a commutative $k$-algebra by the characters of the
irreducible representations with respect to the fundamental highest
weights.
\end{theo}
\begin{proof}
The injectivity relies on the fact that the semisimple regular
elements in $\rG$ form an open dense subset.

The idea to prove the surjectivity is to find a set of
representations \\$\rho:\rG\ra\rGL_{n,k}$ such that the
corresponding set of characters restricted to $\rT$ generates
$k[\rT]^{\rW}$. For more details, one can refer to~\cite{Hum}
3.2,~\cite{St} \S 6 and~\cite{J} Part II, 2.6.
\end{proof}
In the following section, we want to repeat some arguments in the
standard proof to show that those techniques fit quite well for
reductive groups. Moreover, we will generalize these arguments from
fields to $\ent$ when $\rG$ is a split reductive group scheme over
$\ent$, and $\rT$ is a maximal $\ent$-torus of $\rG$.
\end{section}
\begin{section}{The adjoint quotient over $\ent$}
In this section, we will show that a result similar to
Theorem~\ref{0:3} also holds over $\ent$. Namely,
\begin{theo}\label{1:1}
Let $\rG$ be a split reductive $\ent$-group and $\rT$ be a maximal
$\ent$-torus of $\rG$. Let $\rW$ be the Weyl group with respect to
$\rT$. Then the restriction map $i:\ent[\rG]^{\rG}\ra
\ent[\rT]^{\rW}$ is an isomorphism.
\end{theo}
As the first step, we want to generalize the techniques used to
prove Theorem~\ref{0:3}.

\begin{subsection}{The $\rW$-conjugation action on tori}
Let $k$ be a commutative ring and $\rT$ be a split torus over $k$.
Let $\rM$ be the character group of $\rT$ which can be regarded as
an additive group, and $\rM^{\vee}$ be the dual of $\rM$ considered
as $\ent$-module. Let $\mathcal{R}=(\rM,\rM^{\vee},\rR,\rR^{\vee})$
be a reduced root datum with respect to $\rT$ and $\rW$ be the
corresponding Weyl group. Let $\Pi$ be a system of simple roots of
$\rR$. Let $\rM^{+}$ be the set of characters $\lambda$ which
satisfy $(\alpha^{\vee},\lambda)\geq 0$ for all
$\alpha^{\vee}\in\Pi^{\vee}$. A character $\lambda$ is called
\emph{dominant} if $\lambda\in\rM^{+}$.

Here we want to look at $k[\rT]^{\rW}$ more closely. Since
$k[\rT]=k[\rM]$ and $\rW$ permutes $\rM$, we observe that $k[\rT]$
is a $\rW$-permutation module under the conjugation action. Let
$\e^{\lambda}$ be the element in $k[\rT]$ corresponding to $\lambda$
in $\rM$. Then $k[\rT]^{\rW}$ is generated by elements of the form
$\Sym(\e^{\lambda}):=\underset{w\in\rW/\rW_{\lambda}}{\sum}\e^{w(\lambda)}$,
where $\lambda\in\rM$ and $\rW_{\lambda}$ is the stabilizer of
$\lambda$. Since for each $\lambda\in\rM$ we can find a $w\in\rW$
such that $w\lambda$ is in $\rM^{+}$, $k[\rT]^{\rW}$ is a free
$k$-module generated by the set $\{\Sym(\e^{\lambda})|\
\lambda\in\rM^{+}\}$ (ref.~\cite{Bou}, Chap. VI, \S3, Lemma 3),
which in turn means that $k[\rT]^{\rW}$ is determined by the
$\rW$-action on $\rM$ and therefore is stable under arbitrary base
change. We rephrase this fact as a lemma:
\begin{lemm}\label{0:8}
Let $k$, $\rT$ and $\rW$ be defined as above. Then the ring
$k[\rT]^{\rW}$ is a free $k$-module generated by the set
$\{\Sym(\e^{\lambda})\ |\ \lambda\in\rM^{+}\}$ and hence is
determined by the $\rW$-action on $\rM$. In particular, we have
$k[\rT]^{\rW}\otimes k'=k'[\rT]^{\rW}$, for any $k$-algebra $k'$.
\end{lemm}

However, besides the basis $\{\Sym(\e^{\lambda})|\
\lambda\in\rM^{+}\}$, we sometimes need to choose an alternative
basis to simplify our proof. The next two lemmas are useful for this
purpose:

\begin{lemm}\label{0:4}
Let $I$ be an ordered set satisfying the following condition:
\begin{itemize}
  \item[(Min)] Each nonempty subset of $\rI$ contains a minimal element.
\end{itemize}
Let $\rA$ be a commutative ring, $E$ be an $\rA$-module, and
$\{e_i\}_{i\in I}$ be a basis of $E$. Let $\{x_i\}_{i\in I}$ be a
family of elements such that $x_i=e_i+\underset{j<i}{\sum} a_{i,j}\
e_j$, where $a_{i,j}\in\rA$ and only finitely many $a_{i,j}$ are
nonzero. Then also $\{x_i\}_{i\in\rI}$ is a basis of $E$.
\end{lemm}
\begin{proof}
~\cite{Bou}, Chap. VI, \S3, Lemma 4.
\end{proof}

Let us define a partial order on $\rM$ with respect to $\Pi$ by
$\lambda\geq0$ if $\lambda=\underset{s\in\Pi}{\sum} a_{s}s$ where
$a_s$'s are nonnegative integers.
\begin{lemm}\label{0:5}
Given an element $\lambda\in\rM^{+}$, the set
$\rI(\lambda):=\{\mu\in\rM^{+}\ |\ \mu\leq\lambda\}$ is finite.
\end{lemm}
\begin{proof}

If the root datum $\mathcal{R}$ is semisimple, then one can find a
proof of the above lemma in ~\cite{Bou}, Chap. VI, \S3, Prop. 3. For
the reduced root datum $\mathcal{R}$, let
$\mathrm{corad}(\mathcal{R})$ be the coradical of $\mathcal{R}$, and
$\mathrm{ss}(\mathcal{R})$ be the semisimple part of $\mathcal{R}$
(~\cite{SGA3}, Exp. XXI, 6.3.2, 6.5.4, and 6.5.5). The partial order
on $\mathcal{R}$ induces a partial order on
$\mathrm{ss}(\mathcal{R})$, and we define a partial order on
$\mathrm{corad}(\mathcal{R})$ as $x\leq y$ iff $x=y$. Let $p$ be the
canonical isogeny
$p:\mathrm{corad}(\mathcal{R})\times\mathrm{ss}(\mathcal{R})\ra\mathcal{R}$,
and $d$ be the degree of the isogeny. If $\mu\leq\lambda$ in
$\mathcal{R}$, then $d\mu\leq d\lambda$ in
$\mathrm{corad}(\mathcal{R})\times\mathrm{ss}(\mathcal{R})$. Hence
we reduce the case to the semisimple case, and the lemma follows.
\end{proof}
\end{subsection}
\begin{subsection}{Representations associated to dominant weights}
Let $K$ be a field, $\rG$ be a split $K$-reductive group and $\rT$
be a maximal $K$-torus of $\rG$ which splits. Let
$(\rM,\rM^{\vee},\rR,\rR^{\vee})$ be the root datum with respect to
$\rT$, and fix a system of simple roots $\Pi$ of $\rR$. Let $\rB$ be
the Borel subgroup containing $\rT$ determined by $\Pi$ and
$\rB^{-}$ be the opposite Borel subgroup of $\rB$. Let $\rU$ be the
unipotent radical of $\rB$.

In this subsection, we want to associate to each character $\lambda$
which is dominant with respect to $\rB$ a representation
$\rho_\lambda$. As we have mentioned in \\$\S$ 1, this is a crucial
point to prove the surjectivity of $$i:K[\rG]^{\rG}\ra
K[\rT]^{\rW}.$$

In the beginning, we would like to recall some well-known facts over
a field.

Let $\lambda\in\rM^{+}$. Then canonical homomorphism $\rB^{-}\ra\rT$
allows us to view $\lambda$ as a homomorphism from $\rB^{-}$ to
${\rGm}_{,K}$. Thus we can define a $\rB^{-}$-action on
$\rG\times{\rGa}_{,K}$ as
 $a(g,x)=(ga^{-1},\lambda(a)x)$, for all $a\in\rB^{-}(\rA),\ g\in\rG(\rA),$ $x\in{\rGa}_{,K}(\rA)$, where $\rA$ is a $K$-algebra. In this way,
 $(\rG\times{\rGa}_{,K})/\rB^{-}$ becomes a line-bundle over $\rG/\rB^{-}$, and we denote it by $\cL(\lambda)$.
We then identify the global sections
$\rH^{0}(\rG/\rB^{-},\cL(\lambda))$ in a canonical way with the
morphisms of schemes $f:\rG\ra{\rGa}_{,K}$ satisfying
$f(g\beta)=\lambda(\beta)^{-1}f(g)$, for all $g\in\rG(\rA)$,
$\beta\in\rB^{-}(\rA)$, for any $K$-algebra $\rA$ (ref.~\cite{Borel}
4.4). Since $\rG/\rB^{-}$ is projective,
$\rH^{0}(\rG/\rB^{-},\cL(\lambda))$ is finite dimensional. Let
$\rV(\lambda)=\rH^{0}(\rG/\rB^{-},\cL(\lambda))$.

Now we consider the following left regular $\rG$-action on the
functor $\ul{\Hom}_{K-sch}(\rG,{\rGa}_{,K})$ as below. For a
$K$-algebra $\rA$, an element $\sigma\in\rG(\rA)$ and
$f\in\ul{\Hom}_{K-sch}(\rG,{\rGa}_{,K})(\rA)$, we define $l_\sigma
f:\rG_\rA\ra{\rGa}_{,\rA}$ as $(l_\sigma f)(g)=f(\sigma^{-1}g)$ for
any $\rA$-algebra $\rR$ and any $g\in\rG(\rR)$. It is known that
$\rV(\lambda)$ is a $\rG$-module under the left regular $\rG$-action
(ref.~\cite{J} I, 5.12). Let $\rV(\lambda)^{\rU}$ be the subspace
where $\rU$ acts trivially. For $\lambda$ is dominant, we have the
following proposition:
\begin{prop}\label{0:6}
Let $\lambda$ be a character of $\rT$ which is dominant with respect
to the Borel subgroup $\rB$. Then we have the following:
\begin{itemize}
   \item[(1)] $\rV(\lambda)\neq 0$.
   \item[(2)] $\Dim_{K}\rV^{\rU}=1$.
   \item[(3)] $\rT$ acts on $\rV(\lambda)^{\rU}$ by character $\lambda$.
 \end{itemize}
\end{prop}
\begin{proof}
~\cite{J} Part II, 2.2 and 2.6.
\end{proof}
From Proposition~\ref{0:6}, we get the following corollary
immediately:
\begin{coro}\label{0:9}
For each $\lambda$ in $\rM^{+}$, there exists a representation
$$\rho_{\lambda}:\rG\ra\rGL_{n_{\lambda}}$$ such that the character
$\chi_\lambda$ associated to $\rho_\lambda$ restricted to $\rT$
takes the form
$\Sym(\e^{\lambda})+\underset{\underset{\lambda>\mu}{\mu\in\rM^{+},}}{\sum}a_{\mu}\
\Sym(\e^\mu)$.
\end{coro}
\begin{proof}
For each $\lambda\in\rM^{+}$, $\rV(\lambda)$ defined above is
nonempty by Proposition~\ref{0:6} (1). We claim that the
$\rV(\lambda)$'s do the job. Let
$\rV(\lambda)=\underset{\mu}{\oplus}\rV(\lambda)_{\mu}$, where $\rT$
acts on $\rV(\lambda)_{\mu}$ by the character $\mu$. Note that $\rW$
permutes the $\mu$'s, so
$\Dim_{K}\rV(\lambda)_{\mu_1}=\Dim_{K}\rV(\lambda)_{\mu_2}$ if
$\mu_1$, $\mu_2$ are in the same $\rW$-orbit. Let $\mu$ be a maximal
weight of $\rV(\lambda)$. Then $\rV(\lambda)_{\mu}$ is fixed by
$\rU$ (ref.~\cite{J} Part II, 1.19 (7)). By Proposition~\ref{0:6}
(3), $\mu=\lambda$. Therefore, $\lambda$ is the unique maximal
weight in $\rV(\lambda)$ and by Proposition~\ref{0:6} (2),
$\Dim_{K}\rV(\lambda)_{\lambda}=1$. Hence the character
$\chi_{\lambda}$ restricted to $\rT$ takes the form
$\Sym(\e^{\lambda})+\underset{\underset{\lambda>\mu}{\mu\in\rM^{+},}}{\sum}a_{\mu}\
\Sym(\e^\mu)$.
\end{proof}

Now we consider a split $\ent$-group $\rG$ equipped with a maximal
$\ent$-torus $\rT$. We keep all the notations defined above. What we
want to show next is that Corollary~\ref{0:9} is also true over
$\ent$. We start with a lemma.
\begin{lemm}\label{0:1}
Let k be a commutative ring and $\rG$ be an affine $k$-group scheme.
Let $\rV$ be a $\rG$-module and $\Delta:\rV\ra k[\rG]\otimes\rV$ be
the corresponding comodule map. Let
$\rV^{\rG}=\{v\in\rV|\Delta(v)=1\otimes v\}$. Then for any flat
k-algebra $\rA$,
$\rV^{\rG}\otimes_{k}\rA=(\rV\otimes_{k}\rA)^{\rG\times_{k}\rA}$ .
\end{lemm}
\begin{proof}
We refer to~\cite{Sesh} Lemma 2.
\end{proof}
\begin{rema}\label{0:2}
When $\rA$ is not a flat $k$-algebra, the above lemma may be false.
For example, let $k=\ent$, $\rA=\ent/p\ent$, where $p$ is an integer
and $\rG=\spec(k[t])$ defined as the additive group. Let
$\rV=k\oplus k$, and $e_1,e_2$ be its standard basis. Define
$\Delta:\rV\ra k[\rG]\otimes\rV$ as $\Delta(e_1)=1\otimes e_1$ and
$\Delta(e_2)=pt\otimes e_1+1\otimes e_2$. Then $\Delta$ is a
comodule map and $\rV^{\rG}\otimes_{k}\rA=\rA e_1$ while
$(\rV\otimes_{k}\rA)^{\rG\times_{k}\rA}=\rV\otimes_{k}\rA$.
\end{rema}

\end{subsection}
Let $\rV(\lambda)$ be the $\rG$-module
$\rH^{0}(\rG_{{\rat}}/\rB^{-}_{{\rat}},\cL(\lambda)_{{\rat}})$
defined above. Then there is a lattice
$\rN(\lambda)\subseteq\rV(\lambda)$ which is a $\rG$-module
(ref.~\cite{Ser}, \S 2.4, Lemme 2).  By Lemma~\ref{0:1}, we have
$$\rN(\lambda)^{\rU}\otimes_{\ent}\rat\simeq(\rN(\lambda)\otimes_{\ent}{\rat})^{\rU_{{\rat}}}.$$
By Proposition~\ref{0:6},
$\Dim_{{\rat}}(\rN(\lambda)\otimes_{\ent}{\rat})^{\rU_{{\rat}}}=\Dim_{{\rat}}(\rV(\lambda))^{\rU_{{\rat}}}=1$,
so the $\ent$-module $\rN(\lambda)^{\rU}$ is free of rank one and
$\rT$ acts on it by $\lambda$. Hence $\lambda$ is the unique maximal
weight of $\rN(\lambda)$ and
$\rN(\lambda)_{\lambda}=\rN(\lambda)^{\rU}$. Now we get all the
ingredients at hand to prove Theorem~\ref{1:1} over $\ent$.
\begin{proof}[Proof of Theorem~\ref{1:1}]
The injectivity part is the easy one. Note that for a split
reductive $\ent$-group, $\ent[\rG]$ is a free $\ent$-module, so
$\ent[\rG]^{\rG}$ is torsion free and is flat over $\ent$. Therefore
the map
$\ent[\rG]^{\rG}\hookrightarrow\ent[\rG]^{\rG}\otimes_{\ent}\ratc$
is injective. By Lemma~\ref{0:1},
$\ent[\rG]^{\rG}\otimes_{\ent}\ratc=(\ent[\rG]\otimes_{\ent}\ratc)^{\rG_{\ratc}}$.
Now the injectivity follows from Theorem~\ref{0:3}.

For the surjectivity, let $\lambda$ be a dominant character of $\rT$
with respect to the Borel subgroup $\rB$ and $\rho_{\lambda}$ be the
homomorphism $\rho_{\lambda}:\rG\ra\rGL(\rN(\lambda))$ defined as
above. Let $\chi_{\lambda}$ be the character of $\rho_{\lambda}$.
Since $\rN(\lambda)_{\lambda}$ is free of rank one, $\chi_{\lambda}$
restricted to $\rT$ will take the form
$\Sym(\e^\lambda)+\underset{\mu<\lambda,\mu\in\rM^{+}}{\Sigma}n_{\mu}\cdot\Sym(\e^\mu)$,
where $n_{\mu}$ is the rank of $\rN(\lambda)_{\mu}$. By
Lemma~\ref{0:4} and Lemma~\ref{0:5}, the $\chi_{\lambda}$'s
restricted to $\rT$ form a basis of $\ent[\rT]^{\rW}$ and the
surjectivity follows.

\end{proof}

Let $\lambda$ in $\rM^{+}$  and define \emph{the character
associated to the weight $\lambda$} to be  $\chi_\lambda$ as in the
above proof. Then we have the following Corollary:
\begin{coro}~\label{1:6}
Let $\rG$, $\rT$ and $\rW$ be defined as in Theorem~\ref{1:1}. In
case $\rG$ is semisimple simply connected, $\ent[\rG]^{\rG}$ is
freely generated as a commutative $\ent$-algebra by the characters
associated to the fundamental weights. Especially, we have
$\adqG\simeq\aff^{n}_{\ent}$, where $n$ is the rank of $\rG$.
\end{coro}
\begin{proof}
Let $\{\lambda_1,...,\lambda_n\}$ be the fundamental weights of
$\rT$, and $\chi_i$'s be the characters associated to $\lambda_i$'s
respectively. Since $\rG$ is semisimple simply connected, the
fundamental highest weights generate $\rM^{+}$. From the proof of
Theorem~\ref{1:1}, we know that $\lambda:=\sum_i m_i\lambda_i$ is
the unique maximal weight occurring in
$\underset{i=1,...n}{\Pi}\chi_i^{m_i}$ and the monomial
$\e^{\lambda}$ is with coefficient 1. Since $\chi_i$'s are in
$\ent[\rT]^{\rW}$, $\underset{i=1,...n}{\Pi}\chi_i^{m_i}$ takes the
form
$\Sym(\e^{\lambda})+\underset{\mu<\lambda}{\sum}c_\mu\Sym(\e^{\mu})$.
By Lemma~\ref{0:4} and Lemma~\ref{0:5},
$\{\underset{i=1,...n}{\Pi}\chi_i^{m_i}\}_{m_i\in\ent_{\geq0}}$
restricted to $\rT$ form a basis of $\ent[\rT]^{\rW}$ and hence
$\ent[\rG]^{\rG}\simeq\ent[\rT]^{\rT}\simeq\ent[\chi_1,...,\chi_n]$.
The rest of the Corollary follows.

\end{proof}
\end{section}
\begin{section}{Stability under base change}
Let $k$ be a commutative ring and $\rG$ be a reductive $k$-group
with a maximal torus $\rT$. Let $\rW$ be the corresponding Weyl
group of $\rT$. In the previous section, we have proved that
$\adqG\simeq\rT//\rW$ when $k=\ent$. Here, we want to show that this
result holds over an arbitrary commutative ring $k$. To be more
precise, we have the following theorem:
\begin{theo}\label{2.1}
Let $k$ be a commutative ring and $\rG$ be a reductive $k$-group.
Then:
\begin{itemize}
\item [(1)]
For any $k$-algebra $\rA$,
$\rG_{\rA}//\rG_{\rA}\xrightarrow{\sim}({\adqG})_{\rA}$.
\item [(2)] Suppose that $\rG$ has a maximal torus $\rT$. Let $\rW$ be the corresponding Weyl
group of $\rT$. Then $\rT//\rW\xrightarrow{\sim}\adqG$.
\end{itemize}

\end{theo}

Instead of proving Theorem~\ref{2.1} directly, we prove the
following special case first:
\begin{lemm}\label{2.2}
Let $\rA$ be a commutative ring. Let $\rG$ be a split reductive
group over $\ent$ and $\rT$ be a maximal torus of $\rG$. Let $\rW$
be the Weyl group with respect to $\rT$. Then
$\rG_{\rA}//\rG_{\rA}\xrightarrow{\sim}(\adqG)_\rA$ and $
{\rT_\rA//\rW_\rA}\xrightarrow{\sim}\rG_{\rA}//\rG_{\rA}$.
\end{lemm}
\begin{proof}
As we have shown in Theorem~\ref{1:1}, $\rT//\rW\simeq\adqG$ and
hence $(\rT//\rW)_\rA\simeq(\adqG)_\rA$. On the other hand, from
Lemma~\ref{0:8}, we have $(\rT_\rA//\rW_\rA)\simeq(\rT//\rW)_\rA$.
Therefore,
 the isomorphism between ${\rT_\rA//\rW_\rA}$ and $\rG_{\rA}//\rG_{\rA}$ follows from the
 isomorphism between $\rG_{\rA}//\rG_{\rA}$ and $(\adqG)_\rA$. So it is enough to prove $\rG_{\rA}//\rG_{\rA}\xrightarrow{\sim}(\adqG)_\rA$.

 Let $c$ be
the comodule map corresponding to the $\rG$-conjugation action on
$\ent[\rG]$. Define the map $\gamma$ as $\gamma(f)=c(f)-1\otimes f$.
Then we have the following exact sequences:
\begin{center} $\xymatrix {0\ar[r]&
\ent[\rG]^{\rG}\ar[r]& \ent[\rG]\ar[r]^{\gamma}&
 \Img(\gamma)\ar[r]& 0}$,
 $\xymatrix {0\ar[r]&
\Img(\gamma)\ar[r]& \ent[\rG]\otimes_\ent\ent[\rG]\ar[r]&
 \rQ\ar[r]& 0}$.
\end{center}

Since $\ent[\rG]$ is a free $\ent$-module and $\Img(\gamma)$ is a
submodule of $\ent[\rG]\otimes_\ent\ent[\rG]$, $\Img(\gamma)$ is
torsion-free and hence a flat $\ent$-module. Therefore, for an
arbitrary commutative ring $\rA$, we have the following exact
sequence
\begin{center}
$\xymatrix {0=\tor_\ent\ar[r](\rA,\Img(\gamma))&
\rA\otimes\ent[\rG]^{\rG}\ar[r]& \rA[\rG]\ar[r]^{\gamma_{\rA}}&
\rA\otimes\Img(\gamma)\ar[r]& 0 }$.
\end{center}
If we can prove that $\rQ$ is also flat, then the multiplication map
$$\Img(\gamma)\otimes_\ent\rA\ra\rA[\rG]\otimes_\rA\rA[\rG]$$ will be
injective, which in turn means that
$\rA\otimes\ent[\rG]^{\rG}=\rA[\rG]^{\rG}$.

Now take a prime $p$ and consider the following diagram:
\begin{center}
$\xymatrix{
\ent/p\ent\underset{\ent}{\otimes}\ent[\rG]^{\rG}\ar[rr]^{\underset{\sim}{id_p\otimes
i}}\ar[d]_{m} && \ent/p\ent\underset{\ent}{\otimes}\ent[\rT]^{\rW}
\ar[d]^{\wr}_{m} \\
\ent/p\ent[\rG]^{\rG}\ar[rr]^{i_{p}} && \ent/p\ent[\rT]^{\rW}.}$
\end{center}
By Lemma~\ref{0:8}, we have
$$\ent/p\ent[\rT]^{\rW}\simeq\ent/p\ent\otimes_\ent\ent[\rT]^{\rW};$$
by Theorem~\ref{1:1}, we have
$$\ent[\rG]^{\rG}\simeq\ent[\rT]^{\rW};$$
and by Theorem~\ref{0:3}, $i_p$ is an isomorphism. Therefore,
$$\ent/p\ent\underset{\ent}{\otimes}\ent[\rG]^{\rG}\simeq\ent/p\ent[\rG]^{\rG},$$
which implies
$\Img(\gamma)\otimes_\ent\ent/p\ent\ra\ent/p\ent[\rG]\otimes_{\ent/p\ent}\ent/p\ent[\rG]$
is an injection and we get that $\tor_\ent(\ent/p\ent,\rQ)=0$ for
any prime $p$. Hence, $\rQ$ is a flat $\ent$-module and
$\rA\otimes\ent[\rG]^{\rG}=\rA[\rG]^{\rG}$.

\end{proof}

Since every reductive $k$-group is \'etale locally split
(ref.~\cite{SGA3}, Exp. XXII, Cor. 2.3), we can deduce
Theorem~\ref{2.1} from the above lemma.
\begin{proof}[Proof of Theorem~\ref{2.1}]
We prove (1) first. We have a natural morphism
$$m^{a}:\rG_{\rA}//\rG_{\rA}\ra(\rG//\rG)_{\rA}$$ corresponding to the
multiplication $m: \rA\otimes_k k[\rG]^{\rG}\ra \rA[\rG]^{\rG}$.

If we can find an \'{e}tale covering
$\{\spec(\rA_n)\}_n\ra\spec(\rA)$ such that $m^a\times_\rA \rA_n$ is
an isomorphism, then $m^a$ is an isomorphism (~\cite{DG} Chap. III,
\S1, 2.6). Since the reductive group $\rG$  splits \'etale locally
(ref.~\cite{SGA3}, Exp. XXII, Cor. 2.3), we can find an \'{e}tale
covering $\{\spec(k_n)\}_{n\in\rI}\ra\spec(k)$ such that $\rG_{k_n}$
is split with respect to a split maximal torus $\rT_n$ for each
$n\in\rI$. Then over $k_n$, there exists a split reductive
$\ent$-group $\rG_0$ such that $\rG_{k_n}\simeq\rG_{0,k_n}$
(~\cite{SGA3} Exp. XXV, Cor. 1.2 and Exp. XXIII, Cor. 5.10). Let
$\rA_n=\rA\otimes_k k_n$. We want to prove that $m^a\times_\rA
\rA_n:(\rG_{\rA}//\rG_{\rA})_{\rA_n}\ra(\rG//\rG)_{\rA_n}$ is an
isomorphism. Since $k_n$ is flat over $k$, by Lemma~\ref{0:1}, we
have
$$(\rG//\rG)\times_k k_n=\rG_{k_n}//\rG_{k_n},$$ and
$$(\rG_\rA//\rG_\rA)\times_\rA\rA_n=\rG_{\rA_n}//\rG_{\rA_n}.$$
Since $\rG_{k_n}\simeq\rG_{0,k_n}$, we have
$\rG_{A_n}\simeq\rG_{0,\rA_n}$. By Lemma~\ref{2.2}, we have
\begin{align*}
(\rG_\rA//\rG_\rA)\times_\rA\rA_n&\simeq \rG_{\rA_n}//\rG_{\rA_n}\\
&\simeq\rG_{0,A_n}//\rG_{0,A_n}\\
&\simeq(\rG_{0}//\rG_{0})_{\rA_n}\\
&\simeq(\rG_{0}//\rG_{0})_{k_n}\times_{k_n}{\rA_n}\\
&\simeq(\rG_{0,k_n}//\rG_{0,k_n})\times_{k_n}{\rA_n}\\
&\simeq(\rG_{k_n}//\rG_{k_n})\times_{k_n}{\rA_n}.
\end{align*}
On the other hand, we have
$$(\adqG)_{\rA_n}\simeq(\rG//\rG)_{k_n}\times_{k_n}{\rA_n}=(\rG_{k_n}//\rG_{k_n})\times_{k_n}{\rA_n}.$$
Therefore, $m^{a}\times_A A_n$ is an isomorphism and (1) follows.

We prove (2) now. As we have mentioned in the introduction, there is
a natural morphism $\iota:\rT//\rW\ra\adqG$ corresponding to the
restriction map $i:k[\rG]^{\rG}\ra k[\rT]^{\rW}$. As in the proof of
(1), to verify that $\iota$ is an isomorphism, it is enough to prove
that there exists an \'{e}tale covering
$\{\rU_n\}_{n\in\rI}\ra\spec(k)$ such that $\iota\times\rU_n$ is an
isomorphism for all $n\in\rI$.
 Since the reductive group
$\rG$  splits \'etale locally with respect to $\rT$
(ref~\cite{SGA3}, Exp. XXII, Cor. 2.3), we can find an \'etale
covering $\{\spec(k_n)\}_{n\in\rI}\ra\spec(k)$ such that $\rG_{k_n}$
splits with respect to $\rT_{k_n}$ for each $n\in\rI$. Then there
exists a split reductive $\ent$-group $\rG_0$ such that
$\rG_{k_n}\simeq\rG_{0,k_n}$ (~\cite{SGA3} Exp. XXV, Cor. 1.2 and
Exp. XXIII, Cor. 5.10). Then over $k_n$, by Lemma~\ref{2.2}, we have
$$(\ast)\ \rT_{k_n}//\rW_{k_n}\simeq\rG_{k_n}//\rG_{k_n}.$$
Since $k_n$ is flat over $k$, from $(\ast)$, we get
$(\rT//\rW)_{k_n}\simeq(\rG//\rG)_{k_n}.$ Therefore, $\iota$ is an
isomorphism over $k_n$ for each $n$ and hence an isomorphism.

\end{proof}
\end{section}

\begin{section}{Generalized Steinberg's cross-section}

In this section, we let $\rG$ be a semisimple simply-connected group
over a commutative ring $k$. Moreover, we assume that $\rG$ is of
constant type, i.e., there exists a split semisimple group $\rG_0$
over $\ent$ such that $\rG$ is \'{e}tale locally isomorphic to
$\rG_{0,k}$. In this case, $\Isom(\rG_{0,k},\rG)$ is a right
$\underline{\Aut}(\rG_{0,k})$-torsor, and $\rG$ is a form of
$\rG_{0,k}$ twisted by a right $\underline{\Aut}(\rG_{0,k})$-torsor
$\Isom(\rG_{0,k},\rG)$.

Here we want to discuss the adjoint quotient $\adqG$ in this special
case and show that the Steinberg's cross-section can be defined over
arbitrary $k$ in this case.

\begin{subsection}{Adjoint quotients of semisimple simply connected groups}
As we have mentioned in Corollary~\ref{1:6}, if $\rG$ splits, then
the adjoint quotient of $\rG$ is isomorphic to the affine space
$\aff^{n}_{k}$, where $n$ is the rank of $\rG$. In general, since
$\rG$ is locally split in \'{e}tale topology, $\adqG$ is a $k$-form
of $\aff^{n}_{k}$. In the following, we want to show:
\begin{prop}~\label{3:1}
Let k be an arbitrary commutative ring. Let $\rG$ be as in the
beginning of this section. Then $\adqG$ is isomorphic to
$\underset{\rD/\spec k}{\prod}(\aff^{1}_{\rD})$, where $\rD$ is the
Dynkin scheme of $\rG$ (~\cite{SGA3}, Exp. XXIV, 3.2, 3.3), and
$\prod$ stands for the Weil restriction (~\cite{DG}, Chap. I, \S 1,
6.6).
\end{prop}

Before we prove Proposition~\ref{3:1}, we recall some facts about
split semisimple groups.

Let $\rT_0$ be a maximal torus in $\rG_0$, ($\rM_0$, $\rM_0^{\vee}$,
$\rR_0$, $\rR_0^{\vee}$) be the root datum with respect to $\rT_0$
and $\Pi_0$ be a fixed base of $\rR_0$. Let us fix a pinning $\rE_0$
of $\rG_0$ with respect to the chosen base $\Pi_0$ (\cite{SGA3},
Exp. XXIV, 1.0). Let $\underline{\Cent}(\rG_{0,k})$ be the center of
$\rG_{0,k}$ and $\ad(\rG_{0,k})$ be the adjoint group associated to
$\rG_{0,k}$ which is defined by
$\rG_{0,k}/\underline{\Cent}(\rG_{0,k})$. Then we have the following
exact sequence of group schemes which splits (\cite{SGA3}, Exp.
XXIV, Thm. 1.3):
\begin{center}
$(\ast)$ $\xymatrix {1\ar[r]& \ad(\rG_{0,k})\ar[r]&
\underline{\Aut}_{k-grp}(\rG_{0,k})\ar[r]&
 \underline{\Out}(\rG_{0,k})\ar[r]& 1}$.
 \end{center}
Moreover, we can choose a splitting
$s:\underline{\Out}({\rG_{0,k}})\ra\underline{\Aut}{\rG_{0,k}}$ of
$(\ast)$ with respect to $\rE_0$, i.e.,
$\underline{\Out}{\rG_{0,k}}$ acts on $\rE_0$ through $s$. We denote
the image of $s$ as $\underline{\Aut}({\rG_{0,k},\rE_0})$. The
splitting $s$ also allows us to regard
$\underline{\Out}({\rG_{0,k}})$ as the automorphism group of the
Dynkin scheme of $\rG_{0,k}$, because $\rG_{0,k}$ is simply
connected (\cite{SGA3}, Exp. XXIV, 3.6).

Let $\lambda_i$ be the fundamental weight corresponding to the
coroot $\alpha_i^{\vee}\in\Pi_0^{\vee}$ and $\chi_{i}$ be the
character of the fundamental representation associated to the
fundamental weight $\lambda_i$. Let $\Lambda_0$ be the set of
fundamental weights.
Note that $\underline{\Out}{\rG_{0,k}}$ also acts on $\Lambda_0$
through s.

\begin{proof} [Proof of Proposition~\ref{3:1}]
Note that the $\underline{\Aut}_{k-grp}(\rG_{0,k})$-action on
$\rG_0$ induces an $\underline{\Aut}_{k-grp}(\rG_{0,k})$-action on
$k[\rG_{0}]$. Namely, for a $k$-algebra $\rA$,
$\sigma\in{\Aut}_{\rA-grp}(\rG_{0,\rA}),$ $f\in\rA[\rG_{0}]$, we
have $\sigma f\in\rA[\rG_{0}]$ defined as $(\sigma
f)(g)=f(\sigma^{-1}(g))$, for all $g\in\rG_{0,\rA}(\rA')$, for all
$\rA$-algebra $\rA'$. Then by the definition of
$k[\rG_{0}]^{\rG_{0}}$,  $\ad(\rG_{0,k})$ acts on
$k[\rG_{0}]^{\rG_{0}}$ trivially. Therefore,
$\underline{\Aut}_{k-grp}(\rG_{0,k})$ acts on $k[\rG_{0}]^{\rG_{0}}$
through $\underline{\Out}(\rG_{0,k})$, and $\adqG$ is just a form of
$\rG_{0,k}//\rG_{0,k}$ twisted by a right
$\underline{\Out}(\rG_{0,k})$-torsor $\Isomext(\rG_{0,k},\rG)$,
which is isomorphic to $\Isom(\Dyn(\rG_{0,k}),\Dyn(\rG))$ in our
case (~\cite{SGA3}, Exp. XXIV, 3.6). Therefore, we can assume $\rG$
is quasi-split.

Let $j:\Lambda_0\ra\Pi_0$ be the map between sets defined by
$j(\lambda_i)=\alpha_i$. Then $j$ is compatible with the
$\underline{\Out}{\rG_{0,k}}$ action on $\Lambda_0$ and $\Pi_0$,
i.e., $\sigma\lambda_i$ is the fundamental weight corresponding to
the coroot $(\sigma\alpha_i)^{\vee}$. If we regard $\sigma$ as an
element of the symmetric group which permutes $\{1,...,n\}$ and maps
$\sigma\alpha_i$ to $\alpha_{\sigma(i)}$, then
$\sigma\chi_{_i}=\chi_{\sigma({i})}$. Therefore
$\underline{\Out}{\rG_{0,k}}$ acts on $k[\rG_{0}]^{\rG_{0}}\simeq
k[\chi_1,...,\chi_n]$ by permuting the parameters, where $\chi_i$'s
are characters of fundamental representations.

Let $\rS$ be the $\underline{\Out}_{k-grp}(\rG_{0,k})$-torsor
corresponding to
$[\xi]\in\rH^{1}_{et}(k,\underline{\Out}_{k-grp}(\rG_{0,k}))$. Then
$$(\adqG)_{\rS}\simeq(\rG_{0,k}//\rG_{0,k})_{\rS}=(\aff^{1}_\rS)^{\Lambda_0}.
$$ Since $\underline{\Out}_{k-grp}(\rG_{0,k})$ acts on
$(\aff^{1}_\rS)^{\Lambda_0}$ by permuting $\Lambda_0$, and
$\Lambda_0$ and $\Pi_0$ are isomorphic as
$\underline{\Out}_{k-grp}(\rG_{0,k})$-set, we have that
$(\adqG)\simeq\underset{\rD/\spec k}{\prod}\aff^{1}_\rD$.
\end{proof}

For a scheme $\rS$, let $\cO_\rS$ be the structure sheaf of $\rS$,
and $\cE$ be an $\cO_\rS$-module. We define two functors $\cV(\cE)$
and $\cW(\cE)$ over $\rS$-schemes as the following:
\begin{center}
$\cV(\cE)(\rT):=\Hom_{\cO_\rS}(\cE,\cO_\rT)$;\\
$\cW(\cE)(\rT):=\Gamma(\rT,\cE\otimes_{\cO_\rS}\cO_\rT)$,
\end{center}
where $\cE\otimes_{\cO_\rS}\cO_\rT$ means the inverse image of $\cE$
in $\rT$ and $$\Gamma(\rT,\cE\otimes_{\cO_\rS}\cO_\rT)$$ means the
global sections of  $\cE\otimes_{\cO_\rS}\cO_\rT$ over $\rT$. Note
that if $\cE$ is a locally free coherent $\cO_\rS$-module, then
$\cV(\cE)\simeq\cW(\cE^{\vee})$.

With the above notation, we have the following Corollary:
\begin{coro}
Let $k$ be a semi-local ring and $\rG$ be defined as in
Proposition~\ref{3:1}. Then $\adqG\simeq\aff^{n}_k$, where $n$ is
the rank of $\rG$.
\end{coro}

\begin{proof}
Let $\cE$ be the free $\cO_\rD$-module of rank one. Then
$\aff^{1}_\rD\simeq\cV(\cE)\simeq\cW(\cE^{\vee})$ as $\rD$-functors.
Let $\pi:\rD\ra \spec (k)$ be the structure morphism. Then
$\pi_{\ast}\cE^{\vee}$ is just $\cE^{\vee}$  viewed as a $k$-module,
and we have
$$\cW(\pi_{\ast}\cE^{\vee})(\rA)=\cE^{\vee}\otimes_{k}\rA\simeq\cE^{\vee}\otimes_{\cO_\rD}(\cO_\rD\otimes_{k}\rA)=(\underset{\rD/\spec k}{\prod}\cW(\cE^{\vee}))(\rA),$$ for all $k$ algebra $\rA$. Therefore, $\cW(\pi_{\ast}\cE^{\vee})\xrightarrow{\sim}\underset{\rD/\spec k}{\prod}\cW(\cE^{\vee})$.
Since $\pi$ is a finite \'etale morphism, $\cO_\rD$ is a projective
$k$-module of finite type, so is $\pi_{\ast}\cE^{\vee}$. Therefore,
$\pi_{\ast}\cE^{\vee}$ is a free k-module for $k$ semi-local. By
Proposition~\ref{3:1}, $\adqG\simeq\aff^{n}_k$.
\end{proof}
\end{subsection}
\begin{subsection}{ Generalized Steinberg's cross-section}
Let $\rp:\rG\ra\adqG$ be the natural map. Recall that a
cross-section of $\rp$ is a closed subscheme $\rN$ of $\rG$ such
that $\rp$ is  a bijection between functors $\rN$ and $\adqG$.
Suppose $k$ is a perfect field. Then $\adqG$ is isomorphic to
$\aff^{n}_k$, and Steinberg proved that if $\rG$ has a Borel
subgroup, then there exists a cross-section of p. In particular, for
$\rG$ without type $\rA_{2m}$, the Steinberg's cross-section is
contained in $\rG^{reg}$, where $\rG^{reg}$ denotes the open subset
of $\rG$ which consists of regular elements (~\cite{St}, Thm. 1.4,
1.5, and 1.6). In the following, we will show that the similar
result holds for $k$ arbitrary commutative ring and $\rG$
quasi-split without $\rA_{2m}$ components.

We start with the definition of a regular element over an arbitrary
base scheme $\rS$:
\begin{defi}~\label{3.5}
Let $\rG$ be a reductive group with constant type over a scheme
$\rS$. Let $n$ be the rank of $\rG$ and $\rS'$ be an $\rS$-scheme.
An element $g\in\rG(\rS')$ is called regular if its centralizer
$\rC_{\rG}(g)$ is of minimal dimension in all fibers, i.e.
$\Dim(\rC_{\rG}(g))_{s'}<n+1$, for all $s'\in\rS'$.
\end{defi}
\begin{rema}
Keep all the notation above. Since $\Dim(\rC_{\rG}(g))_{s'}\geq n$
for all $s'\in\rS'$, the condition $\Dim(\rC_{\rG}(g))_{s'}<n+1$ is
equivalent to the condition $\Dim(\rC_{\rG}(g))_{s'}=n$.
\end{rema}

Let $\rG$ be as in Definition~\ref{3.5} and $\eta\in\rG(\rG)$ be the
identity map on $\rG$. Define $$\rG^{reg}=\{g\in\rG\mid\
\Dim\rC_{\rG}(\eta)_{g}< n+1 \}.$$ Then by Chevalley's
semi-continuity Theorem (\cite{EGA4}, 13.1.3), $\rG^{reg}$ is an
open subscheme of $\rG$. Moreover, we have $g\in\rG(\rS')$ is
regular if and only if $g\in\rG^{reg}(\rS')$. To see this, note that
we can regard $\rS'$ as a $\rG$-scheme with structure morphism $g$,
and under this morphism, the image of $\eta$ in $\rG(\rS')$ is $g$.
Then we have $\rC_{\rG}(g)=\rC_{\rG}(\eta)\times_{\rG}\rS'$ and
$\rC_{\rG}(g)_{s'}=\rC_{\rG}(\eta)_{g(s')}\times_{g(s')}s'$.
Therefore, $\Dim(\rC_{\rG}(g))_{s'}<n+1$ if and only if
$\Dim\rC_{\rG}(\eta)_{g(s')}< n+1$, which means that $g\in\rG(\rS')$
is regular if and only if $g\in\rG^{reg}(\rS')$.
\begin{rema}
An element $g\in\rG(k)$ is said to be semisimple regular if it is
semisimple regular on each geometrical fibre. Note that for $\rG$ a
reductive group, the definition of a regular element in~\cite{SGA3}
is the definition of a semisimple regular element here (\cite{SCC1},
Exp. 7, Def. 2). The functor of semi-simple regular elements is also
representable by an open subscheme of $\rG$ (\cite{SGA3}, Exp. XIII,
3.1, 3.2).
\end{rema}
\begin{rema}
Given $g\in\rG(\rS')$, since $\rG^{reg}$ is an open subscheme of
$\rG$, $g$ is regular if and only if $g$ is regular at each closed
point of $\rS'$.
\end{rema}
\begin{theo}~\label{3.2}
let $k$ be a commutative ring, $\rG$ be a semisimple simply
connected group of constant type with rank n over $k$. Let $\rG_0$
be the Chevalley group scheme $\rG_0$ associated to $\rG$. If $\rG$
is quasi-split and without components of type $\rA_{2m}$, then there
exists a cross section $C:\underset{\rD/\spec
(k)}{\prod}\aff^{1}_\rD\ra\rG^{reg}$ of $p$, \\i.e., $\rp\circ C$ is
an isomorphism of $\underset{\rD/\spec (k)}{\prod}\aff^{1}_\rD$.
\end{theo}

\begin{proof}
We can fix a quasi-pinning $\rE$ of $\rG$, and a pinning $\rE_0$ of
$\rG_0$. Since $(\rG,\rE)$ is a form of $(\rG_{0,k},\rE_0)$ twisted
by a right $\underline{\Aut}({\rG_{0,k},\rE_0})$-torsor
$\Isomext(\rG_{0,k},\rG)$, we only need to prove the theorem for
$\rG_0$, and show that the section $C_0$ which we construct is
$\underline{\Aut}({\rG_{0,k},\rE_0})$-equivariant.

Let $\rT_0$ be a maximal torus in $\rG_0$, ($\rM_0$, $\rM_0^{\vee}$,
$\rR_0$, $\rR_0^{\vee}$) be the root datum with respect to $\rT_0$
and $\Delta_0$ be a fixed base of $\rR_0$ with respect to the
pinning $\rE_0$. Let $\lieg_0$ be the Lie algebra of $\rG_0$, and
$\liet_0$ be the Lie algebra of $\rT_0$. Let
$\lieg_0=\liet_0\underset{\alpha\in\rR_0}{\oplus}\lieg_0^{\alpha}$
be the decomposition with respect to the adjoint action of $\rT_0$
on $\lieg_0$. For each $\alpha\in\Delta_0$, let
$\rX_\alpha\in\Gamma(\spec k,\lieg_0^{\alpha})^\times$ which is
defined in $\rE_0$. We know that for each $\alpha\in\rR_0$, there is
an unique morphism $\expo_\alpha:\cW(\lieg_0^{\alpha})\ra\rG_0$
which induces the canonical inclusion over the Lie algebra
$\lieg_0^{\alpha}\ra\lieg_0$. Moreover, $\expo_\alpha$ is a closed
immersion (\cite{SGA3} Exp. XXII, Thm. 1.1), and we let $\rU_\alpha$
be the image of $\expo_\alpha$. Let
$p_\alpha:\mathbb{G}_{\mathrm{a},k}\ra\rG_0$ be defined as
$p_\alpha(a)=\expo_\alpha(a\rX_{\alpha})$. For each
$\alpha\in\Delta$, let $w_\alpha$ be the element defined by
$\expo_\alpha(\rX_\alpha)\expo_{-\alpha}(-\rX_\alpha^{-1})\expo_\alpha(\rX_\alpha)$.
Note that the image of $w_\alpha$ in the Weyl group is exactly the
reflection with respect to $\alpha$ (\cite{SGA3} Exp. XXII, 1.5).
Let us number the roots in $\Delta$ in the order such that roots in
the same orbit under $\underline{\Aut}({\rG_{0,k},\rE_0})$ are given
consecutive numbers. Let $\chi_i$ be the fundamental weight
associated to $\alpha_i$ and the $i$-th coordinate of $\aff^{n}_k$
correspond to $\chi_i$.

Define $C_0:\aff^{n}_k\ra\rG_0$ as
$C_0(a_1,...a_n)=\prod^{n}_{i=1}p_{\alpha_i}(a_i\rX_{\alpha_i})w_{\alpha_i}$.

Let $f\in\underline{\Aut}({\rG_{0,k},\rE_0})$. By the definition of
$\underline{\Aut}({\rG_{0,k},\rE_0})$, $f$ permutes the roots in
$\Delta$, and
$f(\expo_\alpha(\rX_\alpha))=\expo_{f(\alpha)}(\rX_{f(\alpha)})$.
Thus, $f$ also permutes the $\sigma_\alpha's$. We can regard $f$ as
an element of the symmetric group which permutes $\{1,...,n\}$ and
maps $\alpha_i$ to $\alpha_{f(i)}$. Since $\rG_0$ has no $\rA_{2m}$
components, there are no edges between $\alpha_i$ and
$\alpha_{f(i)}$ in the Dynkin diagram of $\rG_0$ and
$\rU_{\alpha_i}$ commutes with $\rU_{\alpha_{f(i)}}$. From the way
we number the roots, we have that $\rU_{\alpha_i}$ is next to
$\rU_{\alpha_{f(i)}}$. Therefore,
\begin{align*}
 C_0(f.(a_1,..., a_n))&=C_0(a_{f^{-1}(1)},..., a_{f^{-1}(n)})\\
&=\prod^{n}_{i=1}p_{\alpha_i}(a_{f^{-1}(i)}\rX_{\alpha_i})w_{\alpha_i}\\
&=\prod^{n}_{i=1}p_{\alpha_{f(i)}}(a_i\rX_{\alpha_{f(i)}})w_{\alpha_{f(i)}}\\
&=f(\prod^{n}_{i=1}p_{\alpha_i}(a_i\rX_{\alpha_i})w_{\alpha_i}),
\end{align*}

This shows that $C_0$ is stable under
$\underline{\Aut}({\rG_{0,k},\rE_0})$. Since $\rp\circ C_0$ is an
isomorphism at each point $x\in\spec(k)$ (\cite{St}, Thm. 1.4),
$\rp\circ C_0$ is an isomorphism (\cite{EGA4}, 17.9.5). Moreover,
for a k-algebra $\rR$ and $y\in(\underset{\rD/\spec
(k)}{\prod}\aff^{1}_\rD)(\rR)$, $C_0(y)$ is regular at each fiber
(\cite{St}, 7.9, 7.14), so $C_0$ factors through $\rG_0^{reg}$. The
theorem then follows.
\end{proof}

Recall that for an infinite field $k$, any reductive group has a
semisimple regular element over k (\cite{SGA3}, Exp. XIV, 6.8). If
$k$ is a finite field, Lehrer proves that any semisimple simply
connected group contains a semisimple regular conjugacy class
(\cite{Leh}, Cor. 3.5), and by Lang's Theorem, it implies the
existence of a semisimple regular element over k. Furthermore, for a
semisimple simply connected group over a field, the centralizer of a
semisimple regular element is a maximal torus. Here, we want to show
that the same properties also holds for $k$ a semilocal ring.

\begin{prop}~\label{3.4}
Let $\rG$ be as in the beginning of this section. Suppose that $k$
is a semilocal ring and $\rG$ is quasi-split. Then $\rG$ has a
semisimple regular elements $g\in\rG(k)$. Moreover,
$\underline{\Cent}_{\rG}(g)$ is a maximal torus.
\end{prop}

\begin{proof}
Let $\{s_i\}$ be the set of closed points of $k$ and $\{\kappa_i\}$
be the set of corresponding residual fields. Then from the above
discussion, we know that $\adqG$ has a semisimple regular element
over $\kappa_i$. Let us fix a Borel subgroup $\rB$ of $\rG$ and let
$\rU^+$, $\rU^-$ be the unipotent radical of $\rB$ and the opposite
of $\rB$. Then $\rU^+(k)\ra\Pi_i\rU^+(\kappa_i)$ (resp. $\rU^-$) is
surjective (\cite{SGA3}, Exp. XXII, 5.9.10). On the other hand, for
each $\kappa_i$, we can find a semisimple regular element generated
by $(\rU^+)(\kappa_i)$ and $(\rU^-)(\kappa_i)$ (for finite fields,
see~\cite{St1}, Lemma 64). Therefore, we can find a semisimple
regular element $g\in\rG(k)$. Since every semisimple regular element
is contained in a maximal torus (\cite{SGA3}, Exp. XIII, Thm. 3.1),
$g$ is contained in a maximal torus $\rT$ of $\rG$ which is in turn
contained in $\underline{\Cent}_{\rG}(g)$, and on each geometrical
fibre, they are isomorphic, so $\rT=\underline{\Cent}_{\rG}(g)$
(\cite{EGA4}, 17.9.5).
\end{proof}
\begin{rema}
For $\rG$ as in Proposition~\ref{3.4}. If $\rG$ is without
$\rA_{2m}$ components, then we can show the above Proposition using
Steinberg's cross section. Note that $\adqG$ is isomorphic to
$\underset{\rD/\spec k}{\prod}\aff^{1}$ which is isomorphic to
$\aff^{n}$, so $(\adqG)(k)\ra\Pi_i(\adqG)(\kappa_i)$ is surjective.
Pick $x\in(\adqG)(k)$ which is mapped to the semisimple regular
class of $(\adqG)(\kappa_i)$ for each i. Then by Theorem~\ref{3.2},
there is $g\in\rG(k)$ which is mapped to the semisimple regular
class $x$. This implies $g$ is semisimple regular.
\end{rema}
\begin{rema}
Note that if $\rG$ is a semisimple $k$-group which is not simply
connected, then the connected component of the centralizer of a
semisimple regular element is not necessarily a torus. If
$\underline{\Cent}_{\rG}(g)$ is smooth, then
$\underline{\Cent}_{\rG}(g)^{\circ}$ will be representable
(\cite{SGA3}, Exp. $\mathrm{VI_B}$, Thm 3.10) and hence will be a
torus. However, $\underline{\Cent}_{\rG}(g)$ is not always smooth.
For example, let $k=\compl[[x]]$ and $\rG=\rPGL_{2,k}$. Let $g$ be
the matrix $\left( \begin{array}{cc}
1+x & 0  \\
0 &  -1  \\
\end{array} \right)$. Then $g$ is semisimple regular and contained in a unique maximal torus $\rT$.
Therefore
$\underline{\Cent}_{\rG}(g)\subseteq\underline{\Nor}_{\rG}(\rT)$ and
we can define $ \rW_{g}$ as the closed subgroup scheme of the Weyl
group which fixes $g$. In our case, $\rT$ is a split torus, and the
corresponding Weyl group is a constant group scheme. Suppose
$\underline{\Cent}_{\rG}(g)$ is smooth. Then
$\underline{\Cent}_{\rG}(g)^{\circ}$ is a torus and we have the
following exact sequence:
\begin{center} $\xymatrix {0\ar[r]&
\underline{\Cent}_{\rG}(g)^{\circ} \ar[r]&
\underline{\Cent}_{\rG}(g)\ar[r]&
 \rW_{g}\ar[r]& 0}$.\end{center}
Since $\underline{\Cent}_{\rG}(g)$ and
$\underline{\Cent}_{\rG}(g)^{\circ}$ are flat, $\rW_{g}$ is also
flat (\cite{SGA3}, Exp. $\mathrm{VI_B}$, Prop 9.2 (xi)). So
$\rW_{g}$ is a constant group scheme. However, $\rW_{g}$ is trivial
at the generic point and has an order 2 element $\left(
\begin{array}{cc}
0 & 1  \\
1&  0  \\
\end{array} \right)$ at the closed point, which means $\underline{\Cent}_{\rG}(g)$ cannot be smooth!
\end{rema}
\end{subsection}

\end{section}
\subsection*{Acknowledgements}
Thanks to Brian Conrad, Mathieu Florence, Philippe Gille, and
Patrick Polo for their precious suggestions and comments.

\newpage

\def\refname{R\MakeLowercase{eferences}}

\end{document}